\documentclass{article}
\usepackage{amssymb}
\title{ The Continuity of  $\mathbb{Q}_+$-Homogeneous Superadditive Correspondences \\[0.3cm]}

\author{{Masoumeh Aghajani, \,\,\,Kourosh Nourouzi \thanks{ Corresponding
author } \thanks {e-mail: nourouzi@kntu.ac.ir; fax: +98 21
22853650}
}
\\[0.4cm]
{ \em  Department of Mathematics,  K. N. Toosi University of Technology,}\\
{\em P.O. Box 16315-1618, Tehran, Iran.}\\
}

\newenvironment{proof}{\noindent {\em {Proof .}}}{$\square$
\medskip}

\newtheorem{definition}{Definition}

\newtheorem{ex}{Example}
\newtheorem{thm}{Theorem}

\newtheorem{lem}{Lemma}

\begin{document}

\maketitle \begin{abstract} In this paper we investigate the
continuity of $\mathbb{Q}_+$-homogeneous superadditive
correspondences defined on the cones with  finite basis in real
topological vector spaces. It is also shown that every
superadditive correspondence  between two cones with finite basis
admits  a family of continuous linear selections.

\end{abstract}

\renewcommand{\baselinestretch}{1.1}
\def\thefootnote{ \ }

\footnotetext{{\em} $2010$ Mathematics Subject Classification.
Primary: 54C60, Secondary: 54C65
\\
\indent {\em Key words}: Cone; Superadditive correspondence;
Continuity; selection}

\section{Introduction and Preliminaries}
Let $X$ be a real vector space. A subset  $C$ of $X$ is called  a
cone  if $tx\in C$ for each $x\in C$ and $t>0$. Moreover, a cone
$C$ of $X$ is called a convex cone  if it is convex. If $E$ is a
linearly independent (finite) subset of $X$ such that
$$C=\{\sum_{i=1}^n \lambda_i x_i ;\lambda_i\geq0, x_i\in E, n\in
\mathbb{N}\},$$ then $E$ is said to be a (finite) cone-basis
(briefly, basis) of $C$.

  By a correspondence $\varphi$ from a set $X$ into set $Y$, denoted
  $\varphi:X\twoheadrightarrow Y$, we mean a set-valued function
  $\varphi:X\rightarrow 2^Y\setminus \{\emptyset\}$.

Superadditive correspondences defined on semigroups were
investigated in \cite{wsm} and, in particular, the
Banach-Steinhaus  theorem  of uniform boundedness was extended to
the class of lower semicontinuous and $\mathbb{Q}_+$-homogeneous
correspondences in cones. In \cite{sm},
 it is shown that every superadditive correspondence  from a cone with a basis of a real
 topological vector space into the family of all
 convex, compact  subsets  of a locally convex space admits an additive
 selection. For more information the reader is also referred to \cite{kn, ol, smd}.

 In the present note we first
investigate a more general form of Lemma 1 in \cite{sm}, when the
range of a superadditive correspondence is a finite dimensional
space. We also show that every lower semicontinuous superadditive
correspondence from a cone with a basis of a real
 topological vector space into the family of all
 convex, compact  subsets of a locally convex space admits a
 linear selection. Finally, for every  superadditive correspondence
 defined between two cones with finite basis we find a
family of linear continuous selections.


  We commence some notations and basic concepts.

  Throughout the paper, we assume that $X$ and $ Y$  are real Hausdorff  topological vector spaces and $C$ is  a convex cone in $X$,
  unless otherwise stated. By notations $\varphi:C\rightarrow c(Y)$ and  $\psi:C\rightarrow cc(Y)$ we mean the correspondences $\varphi:C\twoheadrightarrow Y$ and $\psi:C\twoheadrightarrow Y$
 with   compact and
convex, compact values, respectively.
\begin{definition}{\rm
A correspondence $\varphi:C\twoheadrightarrow Y$ is additive if
$\varphi(x+y)=\varphi(x)+\varphi(y)$, for every $x,y\in C$ and is
superadditive if
 $\varphi(x+y)\supseteq\varphi(x)+\varphi(y)$, for every $x,y\in C$.}
\end{definition}

\begin{definition}{\rm
A correspondence $\varphi:C\twoheadrightarrow Y$ is
$\mathbb{Q}_+$-homogeneous if $\varphi(rx)=r\varphi(x)$ for each
$x\in C$ and $r\in \mathbb{Q}_+$. If $\varphi(tx)=t\varphi(x)$ for
each $x\in C$ and $t>0$, it is called positively homogeneous.}
\end{definition}

Every additive and positively homogeneous correspondence
$\varphi:C\twoheadrightarrow Y$ is called linear.



Recall that a correspondence $\varphi:C\rightarrow c(Y)$ is
continuous at $x$ if it is upper and lower semicontinuous at $x$.
If it is continuous at every point $x$ then it is called
continuous on $C$ (see e.g. \cite{alip}).

 Let $(X,d)$ be a metric space and $c(X)$ be the set
of all nonempty compact subsets of $X$. Then the formula
$$\mathfrak{h}(A,B)= \max\{\sup_{a\in A} d(a,B),\sup_{b\in B}
d(A,b)\}\,\,\,\,\,\,\,\,\,\,\ (A,B \in c(X)),$$  defines a
metric, called Hausdorff metric,  on $c(X)$.


\section{Main results}

 We start with the following lemma:
\begin{lem}\label{bound} Let $\{\varphi_\alpha\}_{\alpha\in I}$
be a family of lower semicontinuous and $Q_{+}$-homogeneous
superadditive correspondences $\varphi_\alpha:C\twoheadrightarrow
Y$. If $K$ is a convex and compact subset of $C$ and
$\bigcup_{\alpha\in I} \varphi_\alpha (x)$ is bounded for every
$x\in K$, then $\bigcup_{\alpha\in I} \varphi_\alpha(K)$ is
bounded.
\end{lem}
\begin{proof}
Let $W$ be an open neighborhood of zero in $Y$. There exists an
open balanced neighborhood $U$ of zero such that
$\overline{U}+\overline{U}\subseteq W.$ If
$${{\varphi}_{\alpha}}^{-1}(\overline{U}):=\{x:{\varphi}_{\alpha}(x)\subseteq
\overline{U}\}$$ and $$E:={\cap}_\alpha
{{\varphi}_{\alpha}}^{-1}(\overline{U}),$$ then for every $x\in K$ there
exists $n\in\mathbb{N}$ such that
$${{\varphi}_{\alpha}}(n^{-1}x)\subseteq
U\subseteq\overline{U},\,\,\,\,\,\,\,\,\, (\alpha\in I).$$
Therefore ${n^{-1}x}\in E$, that is, $K\subseteq\cup_{n=1}^\infty
nE.$ Since $K=\cup_{n=1}^\infty {K\cap nE}$ is of the second
category, there exist $n\in \mathbb{N}$, $x_0\in {int}_{K}({K\cap
nE})$ and a  balanced open neighborhood $V$ of zero such that
$$(x_0 + V)\cap K\subseteq K\cap nE\subseteq nE.$$ That is,
$(x_0 + V)\cap K\subseteq nE.$ On the other hand for each $k\in K$
there exists $\lambda>0$ such that $ k\in{x_0 +\lambda V}.$ Since
$K$ is compact, there exist positive numbers
$\lambda_1,\cdots,\lambda_n$ such that $$K \subseteq{{(x_0
+{\lambda_1} V)\cup\cdots\cup(x_0 +{\lambda_n} V)}}.$$ If $\gamma$
is a rational number such that
$\gamma>\max\{\lambda_1,\cdots,\lambda_n,1\}$ , we have
$K\subseteq x_0 +\gamma V.$ For any $x\in K$,
$$z={{\frac{1}{\gamma}}x}+(1-{\frac{1}{\gamma}){x_0}}\in K,$$ and so
$$z\in (x_0+V)\cap K\subseteq nE.$$
 Hence
 $z, x_0\in nE$ and  consequently $
{{\varphi}_{\alpha}}(n^{-1}z)\subseteq\overline{U}$ and
${{\varphi}_{\alpha}}(n^{-1}x_0)\subseteq\overline{U}$ for each
$\alpha$. Since $x+(\gamma-1)x_0 =\gamma z$,
\begin{center}
\begin{tabular}{lll}
${{\varphi}_{\alpha}}(x)$&$\subseteq$&${{\varphi}_{\alpha}}(x)+(\gamma-1){{\varphi}_{\alpha}}(x_0)-(\gamma-1){\varphi_{\alpha}}(x_0)$\\[0.2cm]
&$ \subseteq
$&$\gamma{{\varphi}_{\alpha}}(z)-(\gamma-1){{\varphi}_{\alpha}}(x_0)$\\[0.2cm]
&$ \subseteq $&$\gamma n\overline{U}+(\gamma-1)n\overline{U}$\\[0.2cm]
&$ \subseteq $&$\gamma n\overline{U}+\gamma n\overline{U}$\\[0.2cm]
&$ \subseteq $&$\gamma n{W},$
\end{tabular}
\end{center}
for each $\alpha$. Thus $\bigcup_\alpha
{\varphi_\alpha}(K)\subseteq \gamma nW$,
 that is $\bigcup_\alpha
{\varphi_\alpha}(K)$ is bounded.
\end{proof}

Hereafter we assume that  $C$ is also  with a finite cone-basis,
unless otherwise stated. We shall  denote by $L$ the subspace of
$X$ spanned by the finite cone-basis of $C$.


 The next lemma gives  a general form of Lemma 1 presented
in \cite{sm}.
\begin{thm} \rm{\cite{alip}} Let $C$ be a topological space, $Y$
be a metric space and that $\varphi:C\rightarrow c(Y)$. Let $c(Y)$
be endowed with Hausdorff metric topology. Then the function
$f:C\rightarrow c(Y)$ defined by $f(x)=\varphi(x)$ for each $x\in
C$, is continuous in Hausdorff metric topology if and only if
$\varphi$ is continuous.
\end{thm}
\begin{lem}\label{corol}If $\varphi:C\rightarrow cc(\mathbb{R})$ is a
lower semicontinuous and $Q_{+}$-homogeneous superadditive
correspondence, then it is continuous on $int_L C.$
\end{lem}
\begin{proof}
If $E=\{e_1,\ldots,e_n\}$ is a basis of a cone $C$ in $X$ and
$\varphi:C\rightarrow cc(\mathbb{R})$ is a lower semicontinuous
and $Q_{+}$-homogeneous superadditive correspondence, then the
functions $g:C\rightarrow\mathbb{R}$ and
$h:C\rightarrow\mathbb{R}$ defined by $g(x)=\inf \varphi(x)$ and
$h(x)=\sup \varphi(x)$ for each $x\in C$ are Jensen-convex and
Jensen-concave, respectively. That is $g$ and $h$ satisfy
$$g(\frac{x+y}{2})\leq \frac{g(x)+g(y)}{2}$$
and
$$h(\frac{x+y}{2})\geq \frac{h(x)+h(y)}{2},$$ for all $x,y\in C$,
respectively.
  Since $\varphi$ is compact and convex valued so
$\varphi(x)=[g(x),h(x)]$ for each $x\in C$.  By Lemma \ref{bound},
the correspondence $\varphi$ and consequently  functions $g$ and
$h$ are bounded on some neighborhood of $x_0\in int_L C$.
%
Since, by Lemma 2 in \cite{m}, Jensen-convex functions $g$ and
$-h$ are continuous on $int_L C$, it implies that
  the correspondence $\varphi$ is continuous on $int_L C$.
\end{proof}

In the case that $Y$ is a finite dimensional space we have the
following.
\begin{thm}\label{finite} Let $Y$ be a real
 finite dimensional topological
vector space and $E=\{e_1,\ldots,e_n\}$ be a finite basis of the
convex cone $C$. If $\varphi:C\rightarrow cc(Y)$ is a lower
semicontinuous and $Q_{+}$-homogeneous superadditive
correspondence, then it is positively homogeneous and continuous
on $int_{L}C$.
\end{thm}
\begin{proof} Let
$\{f_1,\ldots,f_m\}$ be a  basis of $Y$ and the norm on $Y$ be
given by $$\|\sum_{i=1}^m \lambda_i f_i\|=\sum_{i=1}^m
|\lambda_i|,$$
 for each real scalar $\lambda_1,\ldots,\lambda_m$.
For $i=1,\ldots,m$, let functions $g_i:C\rightarrow\mathbb{R}$ and
$h_i:C\rightarrow\mathbb{R}$ be defined by  $g_i(x)=\inf \varphi_i
(x)$ and  $h_i(x)=\sup \varphi_i (x)$, respectively, for all $x\in
C$ where
$$\varphi_i (x)=\{\pi_i((\lambda_1,\cdots, \lambda_m)):\, \lambda_1f_1+\cdots+ \lambda_mf_m \in \varphi(x)\}$$
and $\pi_i$ is the $i^{th}$ projection mapping.
If we
define $l:Y\rightarrow \mathbb{R}^m$ by
$$l(\sum_{i=1}^m \lambda_i f_i)=(\lambda_1,\ldots,\lambda_m)$$ we
get $$l(\varphi(x))=\varphi_1 (x)\times\ldots\times
\varphi_m(x),$$ for all $x\in C$. For every $x\in C$ and
$i=1,\ldots,m$ we have $\varphi_i (x)=[g_i (x),h_i(x)]$. Since
$\varphi$ is lower semicontinuous,  the given norm on $Y$ implies
that
 each $\varphi_i$ is lower semicontinuous. By Lemma
\ref{corol}, each correspondence $\varphi_i$ is continuous on
$int_{L} C$ and using Theorem 17.28 in \cite{alip}, it implies
that  $l(\varphi)$ and consequently $\varphi$ is continuous on
$int_{L} C$.

Now we show that $\varphi$ is positively homogeneous. Since
$\varphi$ is compact-valued so $\varphi(tx)=t\varphi(x)$ for $x=0$
and $ t>0$. Let $0\neq x\in int_{L}C$ and $t>0$ be given.  There
is a sequence $(t_n)_n$ in $\mathbb{Q}_+$ such that
$t_n\rightarrow t$. By the continuity of $\varphi$ we have
$\varphi(t_n x)\rightarrow\varphi(tx)$, in the Hausdorff metric
topology on $cc(Y)$. On the other hand $\varphi(t_n x)=t_n
\varphi(x)\rightarrow t\varphi(x)$. Therefore
$t\varphi(x)=\varphi(tx)$ for each $0\neq x\in int_L C$ and $t>0$.
If $x$ is an arbitrary nonzero  element of $C$ and $t>0$, then the
correspondence $\varphi_0:\{\lambda x :\lambda\geq0\}=\langle
x\rangle\rightarrow cc(Y)$ is positively homogeneous on
$int_{M}\langle x\rangle$, where $M=\langle x\rangle-\langle
x\rangle$. Therefore  $\varphi$ is positively homogeneous on $C$.
\end{proof}

The next example shows that Theorem \ref{finite} does not guarantee
the continuity of $\varphi$ on the whole $C$.

\begin{ex} {\rm Define
$\varphi:[0,+\infty)\times[0,+\infty)\rightarrow
\mathbb{R}^2$ by

$$\varphi(x,y)=\left\{%
\begin{array}{ll}
    \{(0,0)\} &  {x\geq 0 , y=0 ;} \\
  \{(t,0):0\leq t\leq x\} & {x\geq 0 , y>0.} \\
\end{array}%
\right.$$ Obviously $\varphi$ is a lower semicontinuous and
positively homogeneous superadditive correspondence. But it is not
continuous on the  $C=[0,\infty)\times[0,\infty)$.}
\end{ex}

In the next result we show that the additive selections can even
be linear.
\begin{thm} \label{gen} Let $C$ be a convex cone with basis and  $Y$ be a real Hausdorff locally
convex topological vector space. Then a superadditive
correspondence $\varphi:C\rightarrow cc(Y)$ admits a linear
selection provided that  $\varphi$ is lower semicontinuous on
$int_L C$.
\end{thm}
\begin{proof} First we show that if $C$ is a convex cone with  finite basis
then the lower semicontinuous superadditive correspondence
$\varphi:C\rightarrow cc(Y)$ admits a linear selection. By Theorem
3 in \cite{sm}, $\varphi$ admits an additive selection
$a:C\rightarrow Y$. Lemma \ref{bound} implies that  $\varphi$ and
consequently additive selection $a$ are  bounded on each bounded
neighborhood $V$ of $x\in int_L C$. By Lemma 2 in \cite{m},
additive function $\Lambda\circ a:C\rightarrow \mathbb{R}$ defined
by $\Lambda\circ a(x)=\Lambda(ax)$ is continuous on $int_L C$ for
each continuous linear functional $\Lambda$ on $Y$. Thus $\Lambda
a(\alpha x)=\Lambda\alpha a(x)$ for each $\alpha>0$ and $x\in
int_L C$. Therefore  $a$ is positively homogeneous on $ int_L C$.
If $x\in C$, then by Lemma 5.28 in \cite{alip},
$\frac{1}{2}(x+y)\in int_L C$ for each $y\in int_L C$ and so
$$\alpha\frac{a(x)}{2}+\alpha\frac{a(y)}{2}=a(\alpha\frac{x+y}{2})=\frac{1}{2}(a(\alpha x)+a(\alpha y)).$$
Since $a(\alpha y)=\alpha a(y)$ for each  $y\in int_L C$, so $a$
is a bounded linear selection of $\varphi$.
Now let $C$ be a convex cone with arbitrary  basis. By Theorem 4
in \cite{sm}, $\varphi$ admits an additive selection
$a:C\rightarrow Y$. Let  $x\in C$ be fixed. There are vectors
$e_1,\cdots,e_n$ and non-negative scalars
$\lambda_1,\cdots,\lambda_n$ such that $x=\sum_{i=1}^n \lambda_i
e_i$. Therefore superadditive correspondence
$\varphi_0:C_0\rightarrow cc(Y)$ with $\varphi_0(z)=\varphi(z)$
for $z\in C_0$, is a lower semicontinuous superadditive
correspondence on  $int_{L_0}C_0$, where $C_0$ and $L_0$ are  the
convex cone  and the linear space generated by
$\{e_1,\cdots,e_n\}$, respectively. Similar to the previous case
$a_0:C_0\rightarrow Y$ defined by $a_0(z)=a(z)$ for $z\in C_0$ is
linear so $a(\alpha x)=\alpha a(x)$ for each $\alpha>0$. That is
$a$ is a linear selection.
\end{proof}

%

\begin{thm}\label{kn.a} Let $Y$ be a real normed space. If $\varphi:C\rightarrow cc(Y)$ is   linear, then
$\varphi$ is automatically continuous.
\end{thm}
\begin{proof}
Let $\sim$ denote the R$\dot{\mbox{a}}$dstr$\ddot{\mbox{o}}$m's
equivalence relation between pairs of  $cc(Y)$ defined by
 $$ (A,B)\sim (D,E)\Leftrightarrow A+E=B+D,\,\,\,\,\,\,\,\,\,\,\,\,(A,B,D,E\in cc(Y)) $$ and $[A,B]$ denote the
 equivalence class of  $(A,B)$ (\cite{rad}).
 The set of all equivalence classes
$\Delta$ with the operations
$$ [A,B]+[D,E]=[A+D,B+E], $$
$$\lambda [A,B]=[\lambda A,\lambda B]\,\,\,\,\,\,\,\,\,\,\,\,\,\,\,(\lambda\geq 0),$$
$$\lambda [A,B]=[-\lambda B,-\lambda A]\,\,\,\,\,\,\,\,\,\,\,\,\,(\lambda< 0),$$ and
the norm $$\|[A,B]\|:=\mathfrak{h}(A,B),$$ constitute a  real
linear normed space (\cite{rad}). The function
$f:C\rightarrow\Delta$ defined by
$$f(x)=[\varphi(x),\{0\}],$$ is linear and  can be
extended to a linear operator $\hat{f}:C-C\rightarrow\Delta$ by
$$\hat{f}(x-y)=f(x)-f(y), \,\,\,\,\,\,\,\,\,\,\,(x,y\in C).$$ Since $C-C$ is of  finite dimension, $\hat{f}$
and consequently $f$ are continuous. Let $x_0\in C$ and $(x_n)$
be a sequence of $C$ converging to $x$. Then
$$\lim_{n\rightarrow
\infty}\mathfrak{h}(\varphi(x_n),\varphi(x_0))=\lim_{n\rightarrow
\infty}\|f(x_n)-f(x_0)\|=0,$$ that is, $\varphi$ is continuous.
\end{proof}

\begin{thm}\label{2} Let $\acute{C}$ be a convex cone with  finite basis
in $Y$ and $\varphi:C\rightarrow cc(\acute{C})$ be a superadditive
correspondence. Then $\varphi$ is continuous on $int_L C$ and
there exists a family of continuous linear functions contained in
$\varphi$.
\end{thm}
\begin{proof}  Let $E=\{e_1,e_2,...,e_n\}$ and
$\acute{E}=\{\acute{e}_1,\acute{e}_2,...,\acute{e}_m\}$ be basis
for $C$ and $\acute{C}$, respectively. Each superadditive
correspondence $\varphi_i:C\rightarrow cc([0,+\infty))$, $
i=1,\cdots,m$ defined as  in the proof of Theorem \ref{finite} is
continuous on $int_L C$ according to  Lemma 1 in \cite{sm}.
 Now, Theorem 17.28 in \cite{alip} implies the continuity of $\varphi$
on $int_L C$.

From  Theorem 1 in \cite{sm}, there exists a minimal
$Q_{+}$-homogeneous superadditive correspondence
$\phi:C\rightarrow cc(\acute{C})$ contained in $\varphi$. Again,
for each $i=1,\cdots,m$ consider the  correspondence
$\phi_i:C\rightarrow cc([0,+\infty))$ as given  in the proof of
Theorem \ref{finite}.  Then by Lemma 1 in \cite{sm}, each
$Q_{+}$-homogeneous superadditive correspondence $\phi_i$ and
consequently $\phi$  are positively homogeneous. By Theorem
\ref{kn.a}, $\phi$ contains a continuous linear correspondence
$\psi:C\rightarrow cc(\acute{C})$ with $ \psi(x)= \sum_{i=1}^n
\lambda_i \phi(e_i)$ for every $x=\sum_{i=1}^n \lambda_i e_i$ in
$C$. Define $l:C-C\rightarrow {[0,\infty)^n}$ and
$\acute{l}:\acute{C}-\acute{C}\rightarrow {[0,\infty)^m}$ by
$$l(\sum_{i=1}^n \lambda_i
e_i)=(\lambda_1,\lambda_2,...,\lambda_n)^T$$ and
$$\acute{l}(\sum_{i=1}^m \gamma_i
\acute{e}_i)=(\gamma_1,\gamma_2,...,\gamma_m)^T,$$ respectively.
The functions $l$ and $\acute{l}$ are linear isomorphisms on $C-C$
and $\acute{C}-\acute{C}$, respectively. Set
$M_i=\acute{l}(\phi({e}_i))$ for $i=1,\cdots,n$ and
$M_\phi=M_1\times M_2\times ...\times M_n$. Then,  $M_\phi$ is a
convex and compact valued multimatrix. Let $x=\sum_{i=1}^n
\lambda_i e_i$ and
$$z\in
\psi(x)=\sum_{i=1}^n \lambda_i \phi(e_i).$$ There are $y_i$'s in
$\phi(e_i)$ such that $z=\sum_{i=1}^n \lambda_i y_i$. If
$y_i=\sum_{j=1}^m \gamma_{ji} \acute{e}_j$, then $z=\sum_{i=1}^n
\lambda_i \sum_{j=1}^m \gamma_{ji} \acute{e}_j$. Putting
$A=[\gamma_{ji}]_{m\times n}$ we have  $A\in M_\phi$ and
$$ \psi(x)\subseteq\{\sum_{j=1}^m \sum_{i=1}^n
\lambda_i \gamma_{ji} \acute{e}_j:\,A=[\gamma_{ji}]_{m\times n}\in
M_\phi\}.$$ On other hand, $$ \psi(x)\supseteq \{\sum_{j=1}^m
\sum_{i=1}^n \lambda_i \gamma_{ji}
\acute{e}_j:\,A=[\gamma_{ji}]_{m\times n}\in M_\phi\}.$$ Therefore
we get
$$\psi(x)= \{\sum_{j=1}^m \sum_{i=1}^n \lambda_i
\gamma_{ji} \acute{e}_j;A=[\gamma_{ji}]_{m\times n}\in
M_\phi\}=\{\acute{l}Al(x)\}_{A\in M_\phi}.$$ Since every matrix
$A$ can be considered as a continuous linear mapping, and
$\varphi(x)\supseteq\psi(x)$, the proof is complete.
\end{proof}


\begin{thebibliography}{10}


\bibitem{kn}  Aghajani M,  Nourouzi K, On the regular cosine family of linear correspondences, Aequationes Math. 83, (2012), no. 3, 215–-221.


\bibitem{alip}  Aliprantis CD,  Border KC, Infinite Dimentional Analysis,
Springer (2005).

\bibitem{m}  Mehdi MR, On convex functions, J. London Math. Soc.
39(1964), 321-326.

\bibitem{ol}  Olko J, Selection of an iteration semigroup of linear set-valued functions, Aequationes Math. 56,
(1998) 157-168.


\bibitem{rad}  R$\dot{\mbox{a}}$dstr$\ddot{\mbox{o}}$m H, An embedding theorem for space of convex sets,
Proc. Amer. Math. Soc. 3 (1952) 165-169.

\bibitem{smd}  Smajdor A, On regular multivalued cosine families, European Conference on Iteration Theory (Muszyna-Ziockie, 1998).
Ann. Math. Sil. No. 13 (1999), 271--280.


\bibitem{wsm}  Smajdor W, Superadditive set-valued functions and
Banach-Steinhause Theorem, Radovi Mat. 3(1987), 203-214.

\bibitem{sm}  Smajdor A, Additive selections of superadditive set-valued
functions, Aequations Mathematicae 39(1990), 121-128.



\end{thebibliography}
\end{document}